\documentclass[12pt]{amsart}
\usepackage[utf8]{inputenc}
\usepackage{amsmath,amsthm,amssymb}
\usepackage{xcolor}

\usepackage[shortlabels]{enumitem}
\usepackage{hyperref}
\usepackage{mathtools}
\usepackage{graphicx}
\newtheorem{theorem}{Theorem}[section]

\newtheorem{defn}[theorem]{Definition}
\newtheorem{definition}[theorem]{Definition}

\newtheorem{lemma}[theorem]{Lemma}
\newtheorem{proposition}[theorem]{Proposition}

\newtheorem{corollary}[theorem]{Corollary}
\mathtoolsset{showonlyrefs}

\newcounter{stepctr}
\newif\ifinsteppage

\newenvironment{steps}{%
	\setcounter{stepctr}{0}%
	\insteppagefalse
}{%
	\par\bigskip
}

\newcommand{\step}[1]{%
	\ifinsteppage
	\par\bigskip
	\fi
	\insteppagetrue
	\refstepcounter{stepctr}%
	\noindent\text{Step \thestepctr:} #1.\par
}

\newcommand{\dem}{\de^{-O(\e)}}
\newcommand{\dep}{\de^{O(\e)}}
\newcommand{\demm}{\de^{-\e}}
\newcommand{\depp}{\de^\e}

\newcommand{\R}{\mathbb{R}}

\newcommand{\Z}{\mathbb{Z}}

\newcommand{\les}{\lessapprox}

\newcommand{\cS}{\mathcal{S}}

\newcommand{\cP}{\mathcal{P}}
\newcommand{\cT}{\mathcal{T}}

\newcommand{\cD}{\mathcal{D}}
\newcommand{\cL}{\mathcal{L}}

\newcommand{\cN}{\mathcal{N}}

\newcommand{\cI}{\mathcal{I}}

\newcommand{\pop}{D_{\mathrm{pop}}}
\newcommand{\spop}{S_{\mathrm{pop}}}
\newcommand{\dy}[2]{\cD_{#1}(#2)}

\newcommand{\lbox}{\underline\dim_{\mathrm{B}}}
\newcommand{\ubox}{\overline\dim_{\mathrm{B}}}

\newcommand{\en}[2]{\mathrm{E}_{#1,#2}}

\newcommand{\cont}[1]{\mathcal{H}^{#1}_{[\delta,\infty)}}

\newcommand{\de}{\delta}
\newcommand{\e}{\epsilon}

\newcommand{\ups}{\upsilon}

\newcommand{\D}{\mathcal{D}}

\newcommand{\dimh}{\dim_\mathrm{H}}

\newcommand{\cY}{\mathcal{Y}}

\newcommand{\cns}[2]{N_{{#1}}({#2})}

\newcommand{\cn}[1]{N_\de({#1})}
\newtheorem*{ack*}{Acknowledgement}

\begin{document}
	\title{A note on the sum-product problem for fractal sets}
	\author{Adam Cushman}
    \author{William O'Regan}

	\thanks{WOR is supported in part by an NSERC Alliance grant administered by Pablo Shmerkin and Joshua Zahl}
	\begin{abstract}
		Utilising recent advances in incidence geometry for balls and tubes, and advances in sum-product theory in the discrete setting, we show that for $0 < s \leq 1/2$ and for any $A \subset \R$ with Hausdorff dimension $s,$ either the upper-box dimension of $AA,$ or the lower-box dimension of $A+A$ must be at least $29s/23.$ We obtain the slightly better bound of \(33 s / 26\) when we replace the sum-set with the smoother difference-set. 
	\end{abstract}
	\maketitle

	\section{Introduction}
    The Erd\H{o}s--Volkmann ring problem \cite{erd} asks whether there are any (Borel) subrings of $\R$ which do not have Hausdorff dimension $0$ or $1.$ In \cite{falconerdistance} Falconer showed (as a corollary of progress on the distance problem) that there are no Borel subrings of Hausdorff dimension strictly between $1/2$ and $1.$ Motivated by the distance problem (and other problems) in \cite{kat} Katz and Tao suggested the discretised ring problem, which is a more quantitative version of ring problem. They asked that, for a discrete set $A,$ which at scale $\de$ resembles a set of Hausdorff dimension $1/2,$ to show that there exists a uniform $c_0 > 0$ so that 
    $$\cn{A+A} + \cn{AA} > \de^{-1/2-c_0}.$$ This was resolved in \cite{bou03} by Bourgain. Independently, using a completely different method, Edgar and Miller \cite{ed} found a clever proof of the (Erd\H{o}s--Volkmann) ring problem.

    It remains open as to how large the constant $c_0$ ought to be. In terms of fractal dimension, the discretised ring problem may ask for $A \subset \R$ with Hausdorff dimension $0 < s < 1$ how large
    \begin{equation}\label{eq.quantity}
    \lbox((A+A) \times AA) 
    \end{equation}
    or $$\dimh ((A+A) \times AA) $$
    would have to be.
    
    When $0 < s \leq 2/3$ the work of \cite{renwang} combined with the argument of Elekes \cite{elekes1997number} show that both are bounded below by $5s/2.$ We are not able to make any improvement of this quantity in this note, however, we remark that by product formulae the quantity \eqref{eq.quantity} may be bounded above by both
\begin{equation}\label{eq.thisone}
\lbox{(A+A)} + \ubox{AA},
\end{equation}
and \begin{equation}
    \ubox{(A+A)} + \lbox{AA}.
\end{equation}
Focussing our attention on \eqref{eq.thisone}, the lower bound of $5s/2$ tells us that
\begin{equation}\label{eq.dimensions}
     \max\{\ubox{(AA)}, \lbox {(A+A)}\} \geq 5s/4.
\end{equation}
The aim of this note is to improve this lower bound to $29s / 23$ when $0 < s \leq 1/2.$ It was obtained (using a different technique) in \cite{doprod} that 
\begin{equation}
     \max\{\lbox{(A/A)}, \ubox {(A+A)}\} \geq 43s/34,
\end{equation}
for $0 < s \leq 1/2.$ 

It is unclear what the correct lower bound ought to be. In the discrete case, for $A \subset \R$ it is conjectured that for all $\e > 0$ there exists $c > 0$ so that
$$|A+A| + |AA| > c|A|^{2-\e}.$$
In the fractal case, certainly the improvement cannot be larger than $\min\{s,1-s\}.$ Further, in \cite{mator} there is an example showing that the increase cannot be larger that $(1-s)/2.$ When $s$ is small, it is unclear whether to expect an improvement of $s.$

For all $0 < s < 1$ there is an example $A \subset \R$ in \cite{mator} with 
$$\lbox(AA) =  \lbox(A+A) = s,$$
and so \eqref{eq.dimensions} cannot be improved much in the sense of replacing $\ubox$ with a finer definition of dimension.
 
As previously mentioned, the results presented here may be viewed as companion results to \cite[Theorem 1.10 and Theorem 1.11]{doprod}. There, similar results were obtained with $AA$ replaced with $A/A$ and $A-A$ replaced with $A
+A,$ and with $33/26$ replaced by the slightly worse $43/34.$ The reader is invited to view \cite[Subsection 1.4]{doprod} for the results presented there.
		\begin{theorem}\label{thm.sumproddim}
		Let $0 < s \leq 1/2.$ Let $A \subset \R$ satisfy $\dimh A = s.$ Then
        \begin{equation}
			\max\{\ubox{(AA)}, \lbox {(A+A)}\}\geq 29s/23.
		\end{equation}
		\begin{equation}
			\max\{\ubox{(AA)}, \lbox {(A-A)}\}\geq 33s/26.
		\end{equation}
	\end{theorem}
    Theorem \ref{thm.sumproddim} will follow from discretised versions below. The deduction is more-or-less identical to the one presented in \cite{doprod} so we omit this here.
    	\begin{theorem}\label{thm.dissum}
		Let $0 < s \leq 1/2$. There exists $C,\delta_0, \epsilon> 0$ so that the following holds for all $0 < \delta < \delta_0.$
		
		\indent Let $A \subset [1/2,1]$ be a $(\delta, s, \de^{-\e})$-set. Then one of the following must hold.
		\begin{equation}
			\cns{\rho}{AA} > \rho^{-29s/23} \text{ for some } \delta < \rho < \delta^{C\e},
		\end{equation}
		\begin{equation}\label{eq.haus}
			\cn{A+A} > \delta^{-29s/23+ C\e}.
		\end{equation}
	\end{theorem}
	\begin{theorem}\label{thm.dis}
		Let $0 < s \leq 1/2$. There exists $C,\delta_0, \epsilon> 0$ so that the following holds for all $0 < \delta < \delta_0.$
		
		\indent Let $A \subset [1/2,1]$ be a $(\delta, s, \de^{-\e})$-set. Then one of the following must hold.
		\begin{equation}
			\cns{\rho}{AA} > \rho^{-33s/26} \text{ for some } \delta < \rho < \delta^{C\e},
		\end{equation}
		\begin{equation}\label{eq.haus}
			\cn{A-A} > \delta^{-33s/26+C\e}.
		\end{equation}
	\end{theorem}
    Note that $29/23 = 1.26086...,$ and $33/26 = 1.26923...$ 
    
The main key ingredient is the recent Sz\'emeredi--Trotter strength incidence inequality for balls and tubes.
    \begin{theorem}
		\label{maininct}\label{thm.product}
		Let $0<s\leq d \leq 1$. Set $\sigma = \min \{s+d,2-s-d\}.$
		Assume $\cT$ is a $(\delta,s,d,K_1,K_2)$-quasi-product set. Consider a shading $Y$ of $\cT$ such that $Y(T)$ is a $(\delta,\sigma,K_3)$-KT set for each $T\in\cT$. Then  we have
		$$\cI(\cT,Y) \les K_3^{\frac13}(K_1K_2)^{2/3} (\delta^{-s-d}\#\cT)^{1/{3}}\# Y(\cT)^{2/3}.$$	
	\end{theorem}
    Here a shading $Y$ is an assignment of a distinct collection of $\de$-balls to each $T \in \cT.$ We write
    $$Y(\cT) = \bigcup_{T \in \cT} Y(T).$$
    We write
    $$\cI(\cT,Y) = \sum_{T \in \cT}\sum_{B \in Y(\cT)} 1_{B \cap T \neq \emptyset}1_{B \in Y(T)}.$$
    The method to obtain Theorem \ref{thm.dis} is quite different from the one in \cite{doprod}, although both rely on Theorem \ref{thm.product}.
Here, we combine ideas present in sum-product literature in the discrete case.
We sketch the main ideas below, focusing on the difference-set case. The sum-set case follows
using standard modifications.

We define the $\de$-\textit{discretised $3$-energy} of two sets $A,B$ by
\begin{equation}
	\en 3 \de (A,B) = \sum_{z \in \de \Z} \Big(\sum_{(a,b) \in A \times B} 1_{|a-b -z| \leq \de}\Big)^3.
\end{equation}
If $A=B$ then we write $\en 3\de (A) = \en 3 \de (A,A).$ 
We let $t \geq s$ be the `discrete Hausdorff dimension' of $A-A.$ 
The strategy can be broken into two distinct parts.
First, we use Theorem \ref{maininct} to upper bound $\en 3 \de (A,A-A)$ and $\en 3 \de (A)$ in terms of $$\# A, \cn {A-A}, \cn{AA} \text{ and } \de^{-t}.$$
Second, we find a lower bound for some multiplicative combination of \(\cn{A-A} \), \( \cn{AA}\), \( \en 3 \de(A)\), and \(\en 3 \de (A,A-A)\)
using a Cauchy-Schwarz argument. This framework is present in the discrete sum-product literature,
where Theorem \ref{thm.product} is now playing the role of the Szemeredi-Trotter theorem. See,
for example, \cite{schoenshkrsumsets}, \cite{rud}, \cite{bloom}, and \cite{csumproduct}.

For the lower bound, we follow \cite{csumproduct} quite closely.
The idea is to produce many solutions to the equation \(p_1 - p_2 = p_3 + O(\delta)\), where the \(p_{i}\)
lie in  a `popular' separated subset of \(A-A\). However, a difficulty arises in our applications
as one can not use the natural definition for the popular set. Indeed, 
the energy estimates obtained using Theorem \ref{thm.product} require the popular set
to be more refined.
To handle this, we introduce two notions of popular set. We have a refined set \(D_{\text{pop}} \)
which is suitable for our estimates obtained using Theorem \ref{thm.product}, and a weaker set \(P\),
which is chosen so that \(\pi ^{-1} (P)\) captures almost all pairs \(A \times A\), where \(\pi(x,y) = x-y\) is
the difference map. We note that \(P\) is not
fit for use in Theorem \ref{thm.product}, but since we never apply energy estimates to \(P\), this does not
prove to be an obstacle. 
We produce solutions
to the equation \(p_1 - p_2 = p_3 + O(\delta)\) where \(p_1,p_2 \in D_{\text{pop}} \) and \(p_3 \in  P\).
Our lower bound, Lemma \ref{lem.energybound1}, follows from a few applications of the H\"older inequality.
\subsection*{Acknowledgements} We thank Shukun Wu for introducing us during WOR's visit to the Univerisity of Indiana, Bloomington. WOR thanks Shukun Wu and the Department of Mathematics for their hospitality, Ciprian Demeter for comments on an earlier version of this note, and Pablo Shmerkin and Joshua Zahl for support.

	\section{Preliminaries}
    For $A \subset \R^d$ let $\dy\de A$ denote the collection of $\de$-dyadic squares which intersect $A.$ We let $N_\de(A) = \# \dy \de A.$ We use $\cN_\de(A)$ to denote the $\de$-neighbourhood of $A.$ 
	\begin{definition}
		Let $0 \leq s \leq d, C, \delta >0.$ We say that $A \subset \R^d$ is a $(\delta,s,C)$-\textit{set} if 
		\begin{equation}
			\cn{A \cap B(x,r)} \leq Cr^s\cn{A}\text{ for all } x \in A, r \geq \delta.
		\end{equation}
		We say that $A \subset \R^d$ is a $(\delta,s,C)$-\textit{KT-set} if 
		\begin{equation}
			\cn{A \cap B(x,r)} \leq C(r/\delta)^s \text{ for all } x \in A, r \geq \delta.
		\end{equation}
		If $C \approx 1$ then we will often drop the $C$ from the notation and simply refer to $(\delta,s)$-sets and $(\de,s)$-KT-sets.
	\end{definition}
    
\begin{defn}\label{KK}
		Given $K_1,K_2\ge 1$ and $s,d\in[0,1]$ we call a  collection $\cT$ of $\delta$-tubes a $(\delta,s,d,K_1,K_2)$-quasi-product set if its direction set $\Lambda$ is a $(\delta,s,K_1)$-KT set and for each $\theta\in\Lambda$, $\cT_\theta$ is a $(\delta,d,K_2)$-KT set.
        \end{defn}

The below definition is standard, and can be found at \cite[Definition 2.3]{demwangszem}, for example.
	\begin{definition}\label{def.uniform}
		Let $\epsilon >0.$ Let $T_\epsilon$ satisfy $T_\epsilon^{-1}\log (2T_\epsilon) = \epsilon.$ Given $0 < \delta \leq 2^{-T_\epsilon},$ let $m$ be the largest integer such that $mT_\epsilon \leq \log (1/\delta).$ Set $T := \log(1/\delta)/m.$ Note that $\delta = 2^{-mT}.$ We say that a $\delta$-separated $A \subset \R^d$ is \textit{$\e$-uniform} if for each $\rho = 2^{-jT}, 0 \leq j \leq m$ and each $P,Q \in \cD_\rho(A),$ we have
		\begin{equation}
			\#(A \cap P) \sim \#(A \cap Q).
		\end{equation}
		A similar definition may be given if $A$ is a collection of $\delta$-balls.
	\end{definition}
	Below is \cite[Lemma 2.4]{demwangszem}.
	\begin{lemma}
		Let $\epsilon > 0.$ There exists $C_\e >0 $ so that the following holds. Let $A \subset \R^d$ be a $\delta$-separated uniform set. For each $\delta \leq \rho \leq 1$ and each $P,Q \in \cD_\rho(A)$ we have 
		\begin{equation}
			C_\epsilon^{-1}\#(A \cap P) \leq \#(A \cap Q) \leq C_\e \#(A \cap P).
		\end{equation}
	\end{lemma}
	Below is \cite[Lemma 2.15]{orpshabc}.
	\begin{lemma}\label{lem.unifsubset}
		Let $\epsilon > 0$ and suppose that $\delta >0 $ is small enough in terms of $\e.$ Let $A \subset \R^d$ be a $\delta$-separated set. Then there is an $\e$-uniform $A_0 \subset A$ with $\#A_0 \gtrsim \delta^{\epsilon}\#A.$ 
	\end{lemma}
Below is \cite[Lemma 2.9]{demwangszem}.
    \begin{lemma}\label{lem.partition}
    Let $\epsilon > 0$ and suppose that $\delta >0 $ is small enough in terms of $\e.$ Let $A \subset \R^d$ be a $\delta$-separated set. Then we may partition $A$ into $\lesssim \demm$ disjoint $\e$-uniform subsets. 
    \end{lemma}

\section{The difference-product case}\label{sect:difference-product}
	Theorem \ref{thm.dis} will follow from the more general \eqref{prob.generalthm} below. The product-set $AA$ may be easily replaced with $A/A$ with only superficial modifications. The reason we are only able to get the result for $A-A$ and not $A+A$ using this method, is due to the fact that Lemma \ref{lem.energybound1} requires the use of symmetry which is present in $A-A,$ but not $A+A.$
In section \ref{sec.sumprod} we will present modifications which will enable us obtain results for $A+A$ and $AA.$  
    
\hfill    

\begin{steps}
\step{Pre-processing}
Let $A \subset [1/2,1]$ be an $\e$-uniform $(\delta,s,\demm)$-set with $\# A < \demm \de^{-s}.$ This also means that $A$ is a $(\de,s,\de^{-2\e})$-KT-set. Let $D \subset A-A$ be $\de$-separated and satisfy $\# D = \cn {A-A}.$ Consider the collection of `popular' intervals
\begin{equation}\label{eq.defofpopularintervals}
    \cD_0 := \left\{I \in \dy \de D : \# \left(\pi^{-1}(I) \cap (A\times A)\right) \geq \# A^2/(100\# D) \right\}. 
\end{equation}
Here $\pi$ is the map $(x,y) \mapsto x-y.$ By Lemma \ref{lem.partition}, after writing $\cD_0$ as a union of disjoint $\e$-uniform sets, of which there are at most $\sim$ $\de^{-\e}, $ by pigeonholing, we find an $\e$-uniform $\cD \subset \cD_0$ so that 

$$\#\bigcup_{I \in \cD} (\pi^{-1}(I)  \cap (A \times A)) \gtrsim \depp \# A^2.$$

Set $\pop := D \cap  \cup \mathcal{D} $.

\step{Dimension of $\pop$}

Now let $\sigma \geq 0$ be the largest exponent so that
\begin{equation}\label{eq.defofsigma}
    \cns \rho \pop \geq \depp \rho^{-\sigma} \text{ for all } \delta \leq \rho \leq 1.
\end{equation}
Since $\pop$ is uniform it is a $(\de,\sigma,\dem)$-set, and therefore a $(\de,\sigma,\dem \de^\sigma \#\pop)$-KT-set. 

Let 

  $$R_A = \{x \in A : \#((x-A) \cap \mathcal{N} _{\delta} (D_{\text{pop}} )) \geq \depp \#A\}.$$
Since $\# R_A \gtrsim \depp \#A,$ choose any $x \in A$ so that 
  $$ \#((x-A) \cap \mathcal{N} _{\delta} (D_{\text{pop}} )) \geq \dep\#A.$$
Therefore $\sigma \geq s.$

The aim is to prove that
\begin{equation}\label{prob.generalthm}
    \dep \de^{-\sigma}\# A^{14} \lesssim \# D^{6} \cn{AA}^{6}.
\end{equation}

The inequality \eqref{prob.generalthm} recovers the (difference-product version of the) result of Ren--Wang \cite[Theorem 1.5]{renwang}
    $$\max\{\cn {A-A},\cn{AA}\} >  \de^{-\frac{5s}{4}+O(\e)}$$
    in the stated range of $s.$

    \step{Elekes' argument}

We recall two results from \cite{doprod}.
The below with $\pop$ replaced with $A+A,$ and $AA$ replaced with $Q$ may be found at \cite[Theorem 5.3]{doprod}. There, $Q$ played a similar role to $\pop$ but in $A/A$ instead of $A-A.$ The proof will be the same with superficial modifications, and so we omit it. This is because the parameter $\sigma$ (which is defined differently in \cite{doprod}) plays no role in this argument.
	\begin{proposition} 
    We have
		\begin{equation}
			\cont{5s/2}(\pop \times AA) \gtrsim \delta^{O(\e)}.
		\end{equation}
	
	\end{proposition}
    Below is essentially Corollary \cite[Corollary 5.4]{doprod}, with $AA$ replacing $Q$ and $\pop$ replacing $A+A.$ Again, the proof involves only superficial changes, so we omit it.
	\begin{corollary}
		\label{cor.fursten} 
		Let $u$ be such that 
		\begin{equation}\label{eq.upperbox}
			\cns{\rho}{AA} < \rho^{-u} \text{ for all } \delta \leq \rho \leq \delta^{O(\epsilon)}.
		\end{equation}
		Then \begin{equation}
			\label{c nh fdvhuivju rtgtig}
      \cont{5s/2 - u}(D_{\text{pop}} ) \gtrsim \dep .
		\end{equation}

	\end{corollary}
\step{Energy upper bounds}
A  result similar to the lemma below in the discrete setting is \cite[Lemma 2]{rud}. When $A=B$, $C=A-A$, $C_1=C_2=1$ and $\#A=\#B=\delta^{-s}$ this inequality is equivalent with 
$$(\#A)^{5/2}\les \cn{AA}\cn{A-A}.$$
\begin{lemma}\label{lem.representations}
		Let $0 < s \leq t < 1, C_1,C_2 \geq 1,$ and suppose that $2s \leq 2- t.$  Let $A \subset [1/2,1]$ and $B \subset [-1,1]$ be $(\delta,s,C_1)$-KT and $(\delta,t,C_2)$-KT, respectively. Let $C$ be a $\de$-separated subset of $A-B.$ Then
        \begin{equation}
            \# \{(a,b,c) \in A\times B \times C: |c- (a-b)| \lesssim \de \} \lesssim \dem  C_1C_2^{2/3}\frac{\delta^{-(s+t)/3}\cn{AA}^{2/3}\#B^{1/3}\#C^{2/3}}{\# A^{2/3}}.
        \end{equation}
		\begin{proof}
			We turn this into an incidence problem, as to exploit Theorem \ref{thm.product}. We introduce the dummy variable $a \in A$:
			\begin{align}\label{eq.step1}
				\sum_{a,a' \in A, b \in B, c \in C} 1_{|c - (a'-b)| \lesssim \delta} &= \sum_{(a',c) \in A \times C}\sum_{(a,b) \in A \times B} 1_{|c - (a'a/a-b)| \lesssim \delta}.
			\end{align}
		This counts (some of the) $\de$-incidences between $AA \times C,$ and a family of tubes parameterised by $A \times B.$ We formalise this below.
		
		 Set $\cP:= AA \times C$ and  $\cL := \{l_{a,b}: a \in A, b \in B\},$ where $l_{a,b}$ is governed by the equation
			\begin{equation}
				y = x/a - b.
			\end{equation}
			Let $\cT$ be the $O(\de)$-neighbourhoods of these lines. Since the lines are parameterised by $A^{-1} \times -B,$ the tubes $\cT$ form a $(\de,s,t,C_1,C_2)$-(quasi)-product set. Here $A^{-1} = \{1/a : a \in A\},$ and the fact that $A^{-1}$ is a $(\de,s,C_1)$-KT-set may be readily checked by observing that the map $x \mapsto x^{-1}$ is bi-Lipschitz with constant $\sim 1$ when restricted to $[1/2,1].$

		Given $(a,b) \in A \times B$ consider $T \in \cT$ with the axial line $l = l_{a,b}.$ These tubes contain the points 
		$$\{(aa',a'-b), : a' \in A, a'-b \in \cN_\de(C)\}.$$
		So define a shading $Y$ on $\cT$ by assigning these points to $Y(T).$ 
			Since $Y(T)$ is an affine transformation of $A$ it is a $(\delta,s,C_1)$-KT-set. 
			
			Returning to \eqref{eq.step1}, we observe
			\begin{equation}
			\sum_{(a',c) \in A \times C}\sum_{(a,b) \in A \times B} 1_{|c - (a'a/a-b)| \lesssim \delta}  \leq  \cI(\cP,\cT,Y).
			\end{equation}
      Since $s \leq \min\{s+t,2-s-t\}$ as per our assumption, we may apply Theorem \ref{thm.product} to obtain
      \begin{equation}\label{eq.energyincidences}
				\cI(\cP,\mathcal{T} ,Y) \lessapprox  C_1C_2^{2/3}\delta^{-(s+t)/3}\cn{AA}^{2/3}\#C^{2/3}\#A^{1/3}\#B^{1/3}.
			\end{equation}
			The result then following after dividing through by $\# A.$
		\end{proof}
	\end{lemma}
	
We use Lemma \ref{lem.representations} to prove the energy bounds below. 
\begin{lemma}\label{lem.energybounds}
  Let \(A \subset [1 / 2 , 1]\) be \((\delta,s,C_1)\)-KT and \(B \subset [-1,1]\) be \((\delta,t,C_2)\)-KT. Suppose that \(2s \leq 2 - t\). We have
\[
    \mathrm{E} _{3,\delta} (A,B) \lesssim  \delta^{- O(\epsilon)} C_1^{3}  C_2 ^{2}\frac{ \delta^{-s - t} N_{\delta} (AA)^{2} \# B}{\# A ^{2}} 
.\]
\end{lemma}

\begin{proof}
Fix a positive integer \(k\) and let \(C_k\) be a maximal \(\delta\)-separated subset of
\[
\left\{I \in \mathcal{D}_\delta\left(A-B\right): \#\left(\pi^{-1}(I) \cap\left(A \times B\right)\right) \sim k\right\} .
\]
Here \(\pi\) is the map \((x, y) \mapsto x-y\). Observing that
\[
\mathrm{E}_{3, \delta}(A,B) \lesssim \sum_{k \text { dyadic }} k^3 \# C_k,
,\]
we look to obtain a bound on \(\# C_k\). 

We apply the above lemma with the triple \(A, D_{\text {pop }}, C_k\) to conclude
\[
    k \# C_{k}  \lesssim  \delta^{-O(\epsilon)} C_1 C_2^{2 / 3} \frac{\delta^{-(s+t) / 3} N_\delta(A A)^{2 / 3} \# B^{1 / 3} \# C_{k} ^{2 / 3}}{\# A^{2 / 3}}
,\]
or
\[
    k^{3} \# C_{k} \lesssim \delta^{- O(\epsilon)} C_1^{3}  C_2 ^{2}\frac{ \delta^{-s - t} N_{\delta} (AA)^{2} \# B}{\# A ^{2}} 
.\]
Summing over all \(\lesssim \delta^{- O(\epsilon)}\) many dyadic levels gives
\[
    \mathrm{E} _{3,\delta} (A,B)\lesssim \delta^{- O(\epsilon)} C_1^{3}  C_2 ^{2}\frac{ \delta^{-s - t} N_{\delta} (AA)^{2} \# B}{\# A ^{2}} 
.\]
\end{proof}
In particular, noting that \(D_{\text{pop}} \) is a \((\delta,\sigma,\delta^{- O(\epsilon)})\)-KT set, we have
\begin{corollary}\label{cor:energybounds}
		\begin{align}
			\en 3 \de (A) &\lesssim \dem \cn{AA}^2 \#A,\\
			\en 3 \de (A,\pop) &\lesssim \dem\frac{\cn{AA}^2 \cn{\pop}^3}{\#A\de^{-\sigma}}.\\
		\end{align}
	\end{corollary}
    \step{Energy lower bounds}
	We now obtain a lower bound on an algebraic combination of the quantities $$ \#D, \#\pop, \en 3 \de(A), \en 3 \de (A,\pop).$$ 
We follow quite closely the argument of \cite[Lemma 1.6]{csumproduct}, mapping solutions of
\[
	(r - a_1) - (r - a_2) = a_2 - a_1
\]
to solutions to \(p_1 - p_2 = p_3 + O(\de) \), where \(p_1,p_2 \in D_{\text{pop}}\), and
\(p_3\) is popular by some weaker notion to be defined shortly. 
We then bound from above the number
of solutions to \(p_1-p_2 = p_3\) by some multiplicative combination of the quantities mentioned previously.

\end{steps}

We introduce the higher \(\delta\)-separated energies, which
proves notationally convenient. 
For any \(A,B \subset \mathbb{R} \), and \(k \in (0,\infty )\), we define the \(k\)-th \(\delta\)-separated additive energy to be
\[
	\mathrm{E} _{k,\delta} (A,B) = \sum _{z \in \delta \mathbb{Z} } \left( \sum _{(a,b)\in A \times B} 1_{\left\lvert a - b - z \right\rvert \leq \delta} 	 \right) ^{k}
.\]
If \(A=B\), we write \(\mathrm{E}_{k,\de}(A)= \mathrm{E}_{k,\de}(A,A)\).
Recall the definition of $\cD$ and $\pop$ in step 1.

\begin{lemma}\label{lem.energybound1}
We have
\[
	\delta^{O(\epsilon)} \# A ^{\frac{15}{2} }\lesssim \mathrm{E} _{3,\delta} (A) \mathrm{E} _{3,\delta} (A, \pop )^{\frac{1}{2} } \# D \# \pop ^{\frac{1}{2} } 
\]
\end{lemma}
\begin{proof}
We will show that the equation \(p_1 - p_2 = p_3 + O(\de)\)
has many solutions, where the \(p_i\) are some popular differences.
We first define a new set of popular intervals for which the requirement is less strict, call
\[
	\mathcal{P} = \left\{ I \in \mathcal{D} _{\delta} \left( D \right) : \# \left( \pi ^{-1} (I) \cap \left( A \times A \right)  \right) \geq \frac{\# A ^{2} \delta^{2\epsilon}}{100 \# D}  \right\} 
.\]
Then let \(P = D \cap \bigcup _{I \in \mathcal{P} } I \).
Note that \(P\) does not need to be as refined as \(D _{ \text{pop} } \), because in our eventual applications,
we do not use incidence estimates for any energy involving \(P\).

We record that
\[
	\sum _{I \in \mathcal{D} _{\delta} (D)} \# \left( \pi ^{-1} (I) \cap \left( A \times A \right)  \right) =  \# A ^{2}
.\]
By the definition of \(\mathcal{P} \), we have
\[
	\sum _{I \in \mathcal{D} _{\delta} (D) \setminus \mathcal{P}} \# \left( \pi ^{-1}(I) \cap (A \times A)  \right) < \# \left( \mathcal{D} _{\delta} (D) \setminus \mathcal{P}  \right) \cdot \frac{\# A ^{2} \delta^{2 \epsilon}}{100 \# D} \leq \frac{\delta^{2 \epsilon}}{100} \cdot \# A ^{2}
,\]
and hence, subtracting from the sum over \(I \in \mathcal{D} _{\delta} (D)\) we have 
\begin{align}\label{eq:sum-over-P}
	\sum _{I \in \mathcal{P}  } \# \left( \pi ^{-1}(I) \cap \left( A \times A \right)  \right) \geq  \# A ^{2}\left( 1 - \frac{\delta^{2 \epsilon}}{100}  \right).
\end{align}
That is, many pairs \((a_1,a_2)\) give rise to these popular differences.

We recall that the rich elements
\[
	R_{A}  = \left\{ x \in A : \# \left( (x - A)\cap \mathcal{N} _{\delta} (D_{\text{pop}} ) \right) \geq \delta^{\epsilon} \# A \right\} 
\]
satisfy \(\# R_{A} \gtrsim \delta^{\epsilon} \# A\). We count the number of tuples
\[
	\left\{ (r,a_1,a_2)\in R_{A} \times A ^{2} : r-a_1 , r-a_2 \in \mathcal{N} _{\delta} (D_{\text{pop}} ),~  a_2 - a_1 \in \mathcal{N} _{\delta} (P)   \right\} 
.\]
We then map these tuples by the truism
\[
	(r- a_1) - (r - a_2) = a_2-a_1
\]
to solutions to \(p_1 - p_2 = p_3 + O(\de)\) where \(p_1,p_2 \in \mathcal{N} _{\delta} (D_{\text{pop}} ) \) and \(p_3 \in \mathcal{N} _{\delta} (P) \).
For each fixed \(r_0 \in R_{A} \), by the definition of \(R_{A} \),
\[
	\# \left\{ (a_1,a_2)\in A ^{2} : r_0-a_1, r_0-a_2 \in \mathcal{N} _{\delta} (D_{\text{pop}} ) \right\} \geq \delta^{2 \epsilon}\# A ^{2}
.\]
Call the set of all such pairs \(S_1\). 

However, by \eqref{eq:sum-over-P}, the number of pairs \((a_1,a_2)\) with \(a_2 - a_1 \in \mathcal{N} _{\delta} (P) \) is 
\[
	\sum _{I \in \mathcal{P} } \# \left( \pi ^{-1} (I) \cap \left( A \times A \right)  \right) \geq   \# A ^{2} \left( 1 - \frac{\delta ^{2 \epsilon}}{100}  \right) 
.\]
Call the set of all such pairs \(S_2\).
We have, by inclusion-exclusion,
\[
	 \#(S_1 \cap S_2)= \#S_1  + \#S_2 -  \#(S_1 \cup S_2 )
,\]
and hence
\[
	\#(S_1 \cap S_2)  \geq \delta^{2 \epsilon}\# A ^{2} + \# A ^{2} \left( 1 - \frac{\delta ^{2 \epsilon}}{100}  \right)  - \# A ^{2} \gtrsim \delta^{2 \epsilon}\# A ^{2}
.\]
That is, for any \(r_0 \in R_{A} \),
\[
	\# \left\{ (a_1,a_2)\in A ^{2} : r_0 - a_1, r_0-a_2 \in \mathcal{N} _{\delta} (D_{\text{pop}} ) ,~ a_2 - a_1 \in \mathcal{N} _{\delta} (P) \right\} \gtrsim \delta^{2 \epsilon}\# A ^{2}
.\]

Thus, taking a union over all \(r \in R_{A} \),
\[
	\# \left\{ (r,a_1,a_2)\in R_{A} \times A^{2} : r - a_1, r-a_2 \in \mathcal{N} _{\delta} (D_{\text{pop}} ) ,~ a_2 - a_1 \in \mathcal{N} _{\delta} (P)  \right\} \gtrsim \delta ^{3 \epsilon} \# A ^{3} 
.\]
Call \(X\) the set of all such triples, and map them by
\[
	f : (r,a_1,a_2)\mapsto (r - a_1, r-a_2)
.\]
Call \(\mathcal{Y} \) the boxes into which \(f\) maps \(X\), i.e.
\[
	\cY = \left\{ (I,J) \in  \cD^2  : \left( I - J \right) \cap \mathcal{N} _{\delta} (P) \neq \varnothing  \right\} 
.\]
Noting that
\[
	f(X) \subset \bigcup _{ B\in \mathcal{Y} } B
,\]
we have by the Cauchy-Schwarz inequality,
\[
	\# X  = \sum _{B \in \mathcal{Y} } \sum _{x \in X} 1_{f(x) \in B}  \leq \# \mathcal{Y} ^{\frac{1}{2} } \# \left\{ (x_1,x_2)\in X^{2} : \left\lvert f(x_1 ) - f(x_2) \right\rvert \lesssim \delta \right\} ^{\frac{1}{2} }
,\]
so substituting our bound on \(X\),
\begin{equation} \label{eq:to-substitute}
	\delta ^{6 \epsilon} \# A ^{6}\lesssim \#\cY \# \left\{ (x_1,x_2)\in X^{2} : |f(x_1) - f(x_2)| \lesssim \delta \right\} .
\end{equation}	

We note that
\[
	\# \cY \sim \sum _{p \in P} \sum _{d_1,d_2 \in \pop }1_{\left\lvert d_1 - d_2 - p \right\rvert \leq \delta} 
.\]
Expanding the \(p\) term, we have by the definition of \(P\) that
\[
	\sum _{p \in P} \sum _{a_1,a_2 \in A} \sum _{d_1,d_2 \in D _{ \text{pop} } } 1_{\left\lvert a_2 - a_1 - p \right\rvert \leq \delta} 1_{\left\lvert d_1 - d_2 - p \right\rvert \leq \delta	} \gtrsim \frac{\delta ^{2 \epsilon}\# A ^{2}}{\# D} \cdot \sum _{p \in P} \sum _{d_1,d_2 \in \pop } 1_{\left\lvert d_1 - d_2 - p \right\rvert \leq \delta	}
,\]
and therefore
\[
	\# \cY \lesssim \frac{\# D}{\delta^{2 \epsilon}\# A^{2}} \sum _{a_1,a_2 \in A} \sum _{d_1,d_2 \in \pop }1_{\left\lvert d_1 - d_2 -a_2 + a_1 \right\rvert \leq \delta} 
.\]
Note that the rightmost sum is
\[
	\sum _{a_1,a_2 \in A} \sum _{d_1,d_2 \in D_{\text{pop}} } 1_{\left\lvert a_1-d_2 - (a_2 - d_1) \right\rvert \leq \delta} \lesssim \mathrm{E} _{2,\delta} (A,D_{\text{pop}} 	)
.\]

Turning our attention to the remaining term, we write  \(x_{i} = (r_{i} ,a_{i} ,a_{i} ')\).
we count
\[
	N := \# \left\{ (x_1,x_2)\in X^{2} : |f(x_1) - f(x_2)| \lesssim \delta \right\} 
.\]
We see that \(\left\lvert f(x_1)-  f(x_2) \right\rvert  \lesssim \delta\) if and only if
\[
	\left\lvert r - a_1 - (r' - a_1' ) \right\rvert  \lesssim \delta, \qquad  \left\lvert r - a_2 - (r' - a_2') \right\rvert \lesssim \delta
,\]
or rearranging,
\[
	\left\lvert r - r' - (a_1 - a_1') \right\rvert  \lesssim \delta, \qquad  \left\lvert r - r' - (a_2 - a_2') \right\rvert \lesssim \delta
.\]
Partitioning the count over values of \(r - r'\),
\begin{align*}
	N & \lesssim \sum _{r,r' \in R_{A} } \left( \sum _{a,a'\in A} 1_{\left\lvert r - r' - (a - a') \right\rvert \lesssim \delta} \right) ^{2} \\
    &\sim  \sum _{x \in \delta \mathbb{Z} } \sum _{r,r' \in R_{A} } 1_{\left\lvert r - r' - x \right\rvert \leq \delta} \left( \sum _{a,a' \in A} 1_{\left\lvert x - (a - a') \right\rvert \lesssim \delta}  \right) ^{2}\\
	& \lesssim \mathrm{E} _{3,\delta} (A).
\end{align*}

Substituting both of our bounds into \eqref{eq:to-substitute}, we obtain
\[
	\delta^{6 \epsilon} \# A ^{6}\lesssim \mathrm{E} _{3, \delta} (A)\cdot \frac{\# D}{\delta^{2 \epsilon}\# A ^{2}} \cdot \mathrm{E} _{2,\delta} (A, D_{\text{pop}} )
.\]
We use Cauchy-Schwarz to interpolate for \(\mathrm{E} _{2,\delta}(A, D_{\text{pop}} ) \), namely we have
\[
	\mathrm{E} _{2,\delta} (A,D_{\text{pop}} )\leq \mathrm{E} _{1,\delta} (A,D_{\text{pop}} )^{\frac{1}{2} } \mathrm{E} _{3,\delta} (A,D_{\text{pop}} ) ^{\frac{1}{2} }
.\]
Noting that
\[
	\mathrm{E} _{1,\delta}(A,D_{\text{pop}} ) = \sum _{x \in \delta \mathbb{Z} } \sum _{a \in A} \sum _{d \in D _{ \text{pop} } } 1_{\left\lvert a - d - x \right\rvert \leq \delta} 	\lesssim \# A \# \pop 	
,\]
we obtain
\[
	\delta^{6 \epsilon} \# A ^{6}\lesssim \mathrm{E} _{3, \delta} (A)\cdot \frac{\# D}{\delta^{2 \epsilon}\# A ^{2}} \cdot \left( \# A \# \pop  \right) ^{\frac{1}{2} } \mathrm{E} _{3,\delta} (A, \pop )	^{\frac{1}{2} }	
,\]
and hence by rearranging we obtain the desired statement.
\end{proof}

We quickly derive \eqref{prob.generalthm}, which is nothing more that combining Corollary \ref{cor:energybounds} and Lemma \ref{lem.energybound1}.
The four bounds from two aforementioned lemmata we seek to combine are: 
            \begin{align}
    			\delta^{O(\epsilon)} \# A ^{\frac{15}{2} }& \lesssim \mathrm{E} _{3,\delta} (A) \mathrm{E} _{3,\delta} (A, \pop )^{\frac{1}{2} } \# D \# \pop ^{\frac{1}{2} } \\
			\en 3 \de (A) &\lesssim \dem \cn{AA}^2 \#A,\\
            			\en 3 \de (A,\pop) &\lesssim \dem\frac{\cn{AA}^2 \#\pop^3 }{\#A\de^{-\sigma}},\\     
            \end{align}
            The combination of the three gives
            \begin{equation}
                \dep \#A^{14}\de^{-\sigma} \lesssim  \# \pop ^{4} \# D^2 \cn {AA}^{6},
            \end{equation}
            and a final rearrangement completes the proof.

    Now we complete the proof of Theorem \ref{thm.dis}. Let $\ups> 0$ be the largest so that 
	\begin{equation}
		\cns \rho {AA} \leq \rho^{- \ups} \text{ for all } \de < \rho < \dep. 
	\end{equation}
    Let $\tau$ be such that
    $\#\D = \de^{-\tau}.$ By Corollary \ref{cor.fursten}, we have $\sigma \geq 5s/2 -\ups.$ Therefore
	\begin{equation}
		    \dep \de^{2\ups}\# A^{33} \lesssim \# D^{12} \cn{AA}^{12},
	\end{equation}
and so,
$33s-O(\e) \leq 12\tau + 14\ups $
    and so
one of $\tau$ or $\ups$ must be larger than $33s/26 - O(\e),$
	completing the proof. 

\section{The sum-product case}\label{sec.sumprod}
We have a corresponding result for the genuine sum-product case. Many steps are identical to the difference-product case after replacing $-$ with $+$ so we do not re-spell out all of the details. We first let \(S \subset A+A\) be \(\delta\)-separated and have \(\# S = N_{\delta} (A+A)\). We define
\(S _{\text{pop}} \) exactly how we did \(D _{\text{pop}} \). In particular there is an
\(\epsilon\)-uniform
\[
	\mathcal{S} \subset \left\{ I \in \mathcal{D} _{\delta} (S) : \#\left( \pi_{+}  ^{-1} (I) \cap (A \times A) \right) \geq \frac{\# A ^{2}}{100 \# S}  \right\} 
\]
where \(\pi_{+} \) is the sum map, i.e. \(\pi_{+} (x,y)= x+y\), for which
\begin{equation} \label{eq:s-pop-fibre}
	\sum _{I \in \mathcal{S} } \# \left( \pi _{+} ^{-1} (I) \cap (A \times A) \right) \gtrsim \delta ^{O(\epsilon)}\# A ^{2}.
\end{equation}
We take \(S _{\text{pop}} = S \cap \cup \mathcal{S} \). Let $\sigma \geq 0$ be the largest exponent so that
\begin{equation}\label{eq.defofsigmasum}
  N_{\rho} (S_{\text{pop}} ) \geq \depp \rho^{-\sigma} \text{ for all } \delta \leq \rho \leq 1.
\end{equation}

By applying Lemma \ref{lem.energybounds}, we have
\begin{corollary}\label{cor:sum-energybounds}
		\begin{align}
			\en 3 \de (A,\spop) &\lesssim \dem\frac{\cn{AA}^2 \cn{\spop}^3}{\#A\de^{-\sigma}},\\
		\end{align}
        \end{corollary}
    Much like Corollary \ref{cor.fursten}, below is essentially Corollary \cite[Corollary 5.4]{doprod}, with $AA$ replacing $Q$ and \(S_{\text{pop}} \) replacing $A+A.$ 

    \begin{corollary}
		\label{cor.fursten-sum} 
		Let $u$ be such that 
		\begin{equation}\label{eq.upperbox}
			\cns{\rho}{AA} < \rho^{-u} \text{ for all } \delta \leq \rho \leq \delta^{O(\epsilon)}.
		\end{equation}
		Then \begin{equation}
			\label{c nh fdvhuivju rtgtig}
			\cont{5s/2 - u}(S_{\text{pop}} ) \gtrsim \dep .
		\end{equation}

	\end{corollary}

The following proposition showcasing the symmetry of \(E_{2,\delta} (A,B)\) will be taken advantage of.

\begin{proposition}\label{prop:e2-cauchy-schwarz}
For any finite \(A,B \subset \mathbb{R} \),
\[
	\mathrm{E} _{2,\delta} (A,B) \sim  \mathrm{E} _{2,\delta} (A,-B)
,\]
and as a consequence,
\[
	\mathrm{E} _{2,\delta} (A,B) \gtrsim \frac{\# A ^{2} \# B ^{2}}{N_{\delta} \left( A + B \right) } 
.\]
\end{proposition}
\begin{proof}
A quick calculation shows
\begin{align*}
	\mathrm{E} _{2,\delta} (A,B) & = \sum _{z \in \delta \mathbb{Z} } \sum _{a,a' \in A} \sum _{b,b' \in B} 1_{\left\lvert a - b - z \right\rvert \leq \delta} 1_{\left\lvert a' - b' - z \right\rvert \leq \delta} \sim  \sum _{a,a' \in A} \sum _{b,b' \in B} 1_{\left\lvert a - b - a' + b' \right\rvert \leq \delta} \\
	& \sim \sum _{z \in \delta \mathbb{Z} } \sum _{a,a' \in A} \sum _{b,b' \in B} 1_{\left\lvert a + b' - z \right\rvert \leq \delta} 1_{\left\lvert a' + b - z \right\rvert \leq \delta} \sim \mathrm{E} _{2,\delta} (A,-B).
\end{align*}

For the second part, let \(X \subset A-B\) be a maximal \(\delta\)-separated subset.
We see by Cauchy-Schwarz that
\[
	\# A \# B \sim \sum _{x \in X } \sum _{(a,b)\in A \times B} 1_{\left\lvert a - b - x \right\rvert \leq \delta} \leq \# X ^{\frac{1}{2} } \mathrm{E} _{2,\delta} (A,B)^{\frac{1}{2} }
,\]
and hence by squaring and rearranging
\[
	\mathrm{E} _{2,\delta} (A,B)\sim \mathrm{E} _{2,\delta} (A,-B)\gtrsim \frac{\# A ^{2} \# B^{2}}{N_{\delta} (A + B)} 
.\]
\end{proof}

To modify our argument to accommodate the sum-set, we use essentially the same idea which is used in the
discrete case.
The issue in passing from the difference-set case to the sum set case
is that one can not modify the projection identity to something involving only sums.
Indeed, we always arrive at something of the form
\[
	(a_1 + a_2) - (a_1 +a_3) = a_2- a_3
.\]
As such, one
must find a difference which can be bounded in terms of the sum-set, and for this we use Proposition \ref{prop:e2-cauchy-schwarz}.
We now map triples \((r_1,r_2,a)\) using the identity
\[
	(r_1 + a) - (r_2 + a) = r_1 - r_2
\]
to solutions to \(p_1 - p_2 = d + O(\delta)\) where \(p_1 \in S_{\text{pop}} \), \(p_2 \in P_{+}\), a weaker popular set to be defined momentarily, and \(d \in P_{\Delta} \),
where \(P_{\Delta} \subset A-A\) is a dyadic level set which supports the energy \(\mathrm{E} _{7 / 4, \delta} (R_1,R_2)\).
We further refine to an \(\epsilon\)-uniform subset of \(P_{\Delta} \) to make use of our energy estimates.

As in the case of the difference-set,
we will need two different notions of richness, corresponding to our two notions of popular sum-sets. The
first rich elements are those which are rich with respect to \(S _{\text{pop}} \), namely
\begin{equation} \label{eq:rich-1}
	R_1 = \left\{ x \in A : \# \left( \left( x + A \right) \cap \mathcal{N} _{\delta} (S_{\text{pop}} )	 \right) \geq c \delta^{\epsilon} \# A \right\},
\end{equation}
where \(c > 0\) is chosen based on the constant in \eqref{eq:s-pop-fibre} so that
\[
	\# R_{1} \gtrsim \delta ^{\epsilon} \# A
.\]

The second notion of popular set is a weaker notion
which we will not need to use any incidence bounds for. Namely, we define the popular
intervals to be
\[
	\mathcal{P}_{+} = \left\{ I \in \mathcal{D} _{\delta} (S) : \# \left( \pi_{+}  ^{-1} (I) \cap (A \times A) \right) \geq \frac{c \# A ^{2} \delta^{\epsilon}	}{4 \# S}  \right\} 
,\]
and we write \(P_{+} = S \cap \bigcup _{I \in \mathcal{P} _{+} } I\). This is the same \(c\) as in \eqref{eq:rich-1}.
The second rich elements are those which are rich with respect to \( P_{+} \), namely
\[
	R_2 = \left\{ x \in A : \# \left( (x +  A)\cap \mathcal{N} _{\delta} (P_{+} )   \right) \geq \left( 1 - \frac{c \delta ^{\epsilon}}{2}  \right) \# A \right\} 
.\]
A quick calculation grants you
\[
	\sum _{I \in \mathcal{P} _{+} } \# \left( \pi _{+} ^{-1} (I)\cap (A \times A) \right) \geq \# A ^{2} \left(  1 - \frac{c \delta^{\epsilon}}{4}  \right) 
.\]
From this, we obtain
\[
	\# R_2 \gtrsim \delta^{\epsilon} \# A
.\]

Similarly to the case of the difference-set, we have set up these rich sets so that many \(a \in A\) have both \(r_1 + a \in \mathcal{N} _{\delta} (S_{\text{pop}} ) \) and \(r_2 + a \in \mathcal{N} _{\delta} (P_{+} ) \)
for arbitrary \(r_1 \in R_1\) and \(r_2 \in R_2\). We are now ready to prove our sum-product lower bound.

\begin{lemma}\label{lem.sumlowerbounds}
We have, for any \(\eta > 0\), 
\[
    \delta^{O(\epsilon)} \delta^{- \sigma / 5    } \# A ^{\frac{11}{2} - \eta }\lesssim_{\eta}  \delta^{- s /5} N_{\delta} (A+A) ^{\frac{11}{5} } N_{\delta} (AA)^{\frac{11}{5} }
.\]
\end{lemma}
\begin{proof}
By Proposition \ref{prop:e2-cauchy-schwarz}, we have
\begin{equation} \label{eq:rr-lower-bound}
	\mathrm{E} _{2,\delta} (R_1,R_2) \gtrsim \frac{\# R_1 ^{2} \# R_2 ^{2}}{\cn{R_1 + R_2}} \gtrsim \frac{\delta^{4\epsilon} \# A ^{4}}{\# S},
\end{equation}
where we have used \(R_{1} + R_{2} \subset A + A\) and \(\# S = N_{\delta} (A+A)\). Thus, it remains to find
an upper bound on \(\mathrm{E} _{2,\delta} (R_1,R_2)\). It will be convenient to first find an
upper bound on \(\mathrm{E} _{\frac{7}{4}  ,\delta} (R_1,R_2)\), and then interpolate this with \(\mathrm{E} _{3,\delta} (A)\).

We use dyadic pigeonholing to find a level set of differences \(R_1-R_2\) supporting the \(7 / 4 \)-energy.
Namely, we write
\begin{align*}
	\mathrm{E} _{\frac{7}{4}  ,\delta} (R_1,R_2) & \sim \sum _{I \in \mathcal{D} _{\delta} (R_1-R_2)} \# \left( \pi ^{-1} (I) \cap (R_1 \times R_2)  \right) ^{\frac{7}{4}  }\\
	& = \sum _{\Delta \text{ dyadic} } \sum _{I \in \mathcal{P} _{\Delta} } \# \left( \pi ^{-1} (I) \cap (R_1 \times R_2) \right) ^{\frac{7}{4}  },
\end{align*}
where
\[
	\mathcal{P} _{\Delta} = \left\{ I \in \mathcal{D} _{\delta} (R_1 - R_2) : \# \left( \pi ^{-1} (I) \cap \left( R_1 \times R_2 \right)  \right) 	\in [\Delta,2\Delta) \right\} 
.\]	
Noting that for any interval \(I\),
\[
	\# \left( \pi ^{-1} (I) \cap (R_1 \times R_2) \right) \leq \# A ^{2}
,\]
it follows that there are
\[
	\lesssim \log \# A \lesssim _{\eta} \# A ^{\eta}
\]
many dyadic levels, and hence there is \(\Delta \in \mathbb{R} \) for which
\[
	\mathrm{E} _{\frac{7}{4} ,\delta} (R_1,R_2) \lesssim _{\eta} \# A ^{\eta} \sum _{I \in \mathcal{P} _{\Delta} } \#\left( \pi ^{-1} (I) \cap (R_1 \times R_2) \right) ^{\frac{7}{4}  } \sim \# A^{\eta} \Delta ^{\frac{7}{4}  } \# \mathcal{P} _{\Delta} 
.\]

Call \(P_{\Delta} = D \cap \bigcup _{I \in \mathcal{P} _{\Delta} } I\), where we recall that \(D\) is \(\delta\)-separated with \(\# D = N_{\delta} (A-A)\). 
We refine further appealing to Lemma \ref{lem.unifsubset} for the existence of
an \(\epsilon\)-uniform \(P_{\Delta} '\), with \(\# P_{\Delta} '\gtrsim \delta^{O(\epsilon)} \# P_{\Delta}  \).
Using the \(\epsilon\)-uniformity, we find that \(P_{\Delta} '\) is fit for energy estimates,
as

\begin{proposition}\label{prop:pdeltaKT}
\(P_{\Delta} '\) is a \((\delta ,s , \delta^{-O (\epsilon)} \# A / \Delta)\)-KT set. 
\end{proposition}
\begin{proof}
Let \(\mathcal{P} _{\Delta} '\) be those intervals of \(\mathcal{P} _{\Delta} \) which have intersection
with \(P_{\Delta} '\). We have
\[
	\sum _{I \in \mathcal{P} _{\Delta} '}  \#  \left( \pi ^{-1} (I) \cap A \times A \right) \geq  \sum _{I \in \mathcal{P} _{\Delta} '} \#  \left( \pi ^{-1} (I) \cap R_1 \times R_2 \right) \thicksim \Delta \# \mathcal{P} _{\Delta} '
,\]
so letting \(R_0\) be the rich set
\[
	R_0 = \left\{ x \in A  : \# \left( (x - A) \cap \mathcal{N} _{\delta} (P_{\Delta} ') \right) \gtrsim \frac{\Delta \# \mathcal{P} _{\Delta} '}{\# A} \right\}  
,\]
we have
\[
	\# R_0 \gtrsim \frac{\Delta \# \mathcal{P} _{\Delta}' }{\# A}
.\]
Since \(R_0\) is nonempty, let \(x_0 \in R_0\) be such that
\[
	\# \left( (x_0-A)\cap N_{\delta} (P_{\Delta} ') \right) \gtrsim \frac{\Delta \# \mathcal{P} _{\Delta} '}{\# A } \gtrsim \delta^{O(\epsilon)}\frac{\Delta \# P_{\Delta} }{\# A}    
.\]

Since \(P_{\Delta} '\) is \(\epsilon\)-uniform and \(\delta\)-separated, for any \(\delta \leq \rho \leq 1\),
\[
	N_{\delta} \left( P_{\Delta} ' \cap B(x,\rho) \right) N_{\rho} (P_{\Delta} ') \thicksim \# P_{\Delta} '
.\]
Note that
\[
	N_{\rho} (P_{\Delta} ') \gtrsim  N_{\rho} \left( (x_0 - A)\cap \mathcal{N} _{\delta} (P_{\Delta} ') \right) 
,\]
and since \(x_0 - A\) is a \(\delta\)-separated \((\delta,s, \delta^{-O(\epsilon)})\)-KT set, the number of
points of \((x_0 - A)\cap \mathcal{N} _{\delta} (P_{\Delta} ') \) in a \(\rho\) ball is
\[
	(x_0 - A)\cap \mathcal{N} _{\delta} (P_{\Delta} ')\cap B(x,\rho)  \lesssim \delta^{-O(\epsilon)} \left( \frac{\rho}{\delta}  \right) ^{s}
,\]
and hence the number of such balls is 
\[
	N_{\rho} \left( (x_0 - A)\cap \mathcal{N} _{\delta} (P_{\Delta} ') \right) \gtrsim \frac{\# \left( x_0 - A \right) \cap \mathcal{N} _{\delta} (P_{\Delta} ')}{\delta^{-O(\epsilon)} \rho^{s} \delta^{-s}}    \gtrsim \delta^{O(\epsilon)} \left( \delta^{s} \frac{\Delta \# P_{\Delta} }{\# A}  \right) \rho ^{-s}
,\]
and hence
\[
	N_{\delta} \left( P_{\Delta} ' \cap B(x,\rho) \right) \lesssim \frac{\# P_{\Delta} '}{N_{\rho} (P_{\Delta} ')   } \lesssim \delta^{- O(\epsilon)} \frac{\# A}{\Delta}  \left( \frac{\rho}{\delta}  \right) ^{s}
.\] 
\end{proof}
Because \(- P_{\Delta} '\) is a \((\delta,s,\delta^{-O(\epsilon)} \# A / \Delta)\)-KT set, we may apply Lemma \ref{lem.energybounds} to give
\[
	\mathrm{E} _{3,\delta} (A,- P_{\Delta} ')\lesssim \delta^{-O(\epsilon)} \frac{\delta^{-2s} N_{\delta} (AA)^{2} \# P_{\Delta} }{\Delta^{2}} 
.\]
We will use this momentarily.
We proceed now as we did in the case of the difference-set, the change being that now we use the
truism
\[
	(r_1 + a) - (r_2 + a) = r_1 - r_2
,\]
to find solutions to the equation \(p_1 - p_2 = p_3 + O(\delta)\) where \(p_1 \in \mathcal{N} _{\delta} (S_{\text{pop}} ) \), \(p_2 \in \mathcal{N} _{\delta} (P_{+}) \), \(p_3 \in \mathcal{N} _{\delta} (P_{\Delta}' )\). 

Carrying this out, we map the set of triples
\[
	X = \left\{ (r_1,r_2,a)\in R_1 \times R_2 \times A : r_1 + a \in \mathcal{N} _{\delta} (S_{\text{pop}} ),~ r_2 + a \in \mathcal{N} _{\delta} (P_{+} ),~ r_1 - r_2 \in \mathcal{N} _{\delta} (P_{\Delta}' ) \right\} 
\]
by
\[
	f : (r_1,r_2,a) \mapsto (r_1 + a ,~ r_2 + a)
.\]
Let \(\mathcal{Y} \) be those intervals into which \(f\) maps \(X\), that is
\[
	\mathcal{Y} = \left\{ (I,J)\in\cS \times \cP_{+}  	: \left( I - J \right) \cap \mathcal{N} _{\delta} (P_{\Delta} ') \neq \varnothing  \right\} 	
,\]
where we note that
\[
	f(X) \subset \bigcup_{B \in \cY} B
.\]

We first obtain a lower bound on \(\# X\).
We choose the rich elements \((r_1,r_2)\) to be all those whose difference is in our dyadic level.
By the definition of \(\mathcal{P} _{\Delta} '\), 
\[
	\Delta \# \mathcal{P} _{\Delta}' \leq \sum _{I \in \mathcal{P} _{\Delta}' } \# \left( \pi ^{-1} (I)\cap (R_1 \times R_2) \right) \leq  \# \left\{ (r_1,r_2)\in R_1 \times R_2 : r_1 - r_2 \in \mathcal{N} _{\delta} (P_{\Delta}' )  \right\} 
.\]
We then choose all the \(a\) to be those which make \(r_1 + a \in \mathcal{N} _{\delta} (S_{\text{pop}} ) \) and \(r_2 + a \in \mathcal{N} _{\delta} (P_{+} 	) \).
For a fixed pair \((r_1,r_2)\in R_1 \times R_2\) by the definition of \(R_1\),
\[
	\# \left\{ a \in A : r_1 + a \in \mathcal{N} _{\delta} (S_{\text{pop}} ) \right\} \geq c \cdot \delta^{\epsilon}\# A
.\]
Call the set of all such \(a\) to be \(A_1\). Then, by the definition of \(R_2\),
\[
	\# \left\{ a \in A : r_2 + a \in \mathcal{N} _{\delta} (P_{+} )  \right\} \geq \left( 1 - \frac{c \delta ^{\epsilon}}{2}  \right) \# A
.\]
Call the set of all such \(a\) to be \(A_2\).
Using inclusion-exclusion,
\[
	\# \left( A_1 \cap A_2 \right) = \# A_1 + \# A_2 - \# \left( A_1 \cup A_2 \right) 
,\]
and hence
\[
	\# \left( A_1 \cap A_2 \right) \geq c \delta^{\epsilon}\# A + \# A \left( 1 - \frac{c \delta^{\epsilon}}{2}  \right) - \# A \gtrsim \delta^{\epsilon} \# A
.\]
Taking a union over all such \((r_1,r_2)\in R_1 \times R_2\), we find that
\[
	\# X \gtrsim \delta^{\epsilon} \Delta \# \mathcal{P} _{\Delta}' \# A \gtrsim \delta^{O(\epsilon)} \Delta \# \mathcal{P} _{\Delta}  \# A
.\]

By a nearly identical argument to that of the difference-set portion,
\[
	\# \left\{ (x_1,x_2)\in X^{2} : \left\lvert f(x_1) - f(x_2) \right\rvert \lesssim \delta \right\} \lesssim \mathrm{E} _{3,\delta} (A)
.\]
Thus, again by Cauchy-Schwarz, we obtain
\begin{equation} \label{eq:sum-post-CS}
	\delta^{O(\epsilon)} \Delta^{2} \# \mathcal{P} _{\Delta} ^{2}\# A ^{2} \lesssim \mathrm{E} _{3,\delta} (A) \# \mathcal{Y} .
\end{equation}

With \(\mathcal{Y} \) in terms of \(P_{\Delta} '\) instead of \(P_{\Delta} \), we may expand the \(\mathcal{P} _{+} \) term
and then use incidence estimates for the remaining terms.
We see that
\[
	\# \mathcal{Y}  \sim \sum _{p_1 \in S_{\text{pop}}   } \sum _{p_2 \in P_{+} } \sum _{p_3 \in P_{\Delta} ' }  1_{\left\lvert p_1 - p_2 - p_3 \right\rvert \leq \delta}
.\]
Expanding \(p_2\), we note that by the definition of \(\mathcal{P} _{+}  \) we have
\[
	\sum _{p_1 \in S_{\text{pop}} } \sum _{p_2 \in P_{+} } \sum _{p_3 \in P_{\Delta}' } \sum _{a_1,a_2 \in A} 1_{\left\lvert a_1 + a_2 - p_2 \right\rvert \leq \delta} 1_{\left\lvert p_1-p_2-p_3 \right\rvert \leq \delta} \gtrsim \frac{\# A ^{2} \delta^{\epsilon}}{\# S} \cdot \# \mathcal{Y} 
,\]
so eliminating the \(P_{+} \) term gives
\[
	\# \mathcal{Y}  \lesssim \frac{\# S}{\# A ^{2}\delta^{\epsilon}} \sum _{p_1 \in S_{\text{pop}} } \sum _{p_3 \in P_{\Delta}' } \sum _{a_1,a_2 \in A} 1_{\left\lvert p_1-p_3-a_1 - a_2 \right\rvert \leq \delta} 
.\]
We group this as \((a_1 - p_1) + (a_2 + p_3)\) to obtain that
\[
	\# \mathcal{Y}  \lesssim \frac{\# S}{\# A ^{2} \delta^{\epsilon}} 	\sum _{z \in \delta \mathbb{Z} } \left( \sum _{a_2 \in A} \sum _{p_3 \in P_{\Delta}' }  1_{\left\lvert a_2 + p_3 + z \right\rvert \leq \delta}  \right) \left( \sum _{a_1 \in A} \sum _{p_1 \in S _{\text{pop}} 	}  1_{\left\lvert a_1 - p_1 - z \right\rvert \leq \delta} \right) 
,\]
and hence by using H\"older we get
\[
	\# \mathcal{Y}  \lesssim \frac{\# S}{\# A ^{2}\delta^{\epsilon}} \mathrm{E} _{\frac{3}{2} ,\delta} (A,- P_{\Delta}' )^{\frac{2}{3} } \mathrm{E} _{3,\delta} (A, S_{\text{pop}} )^{\frac{1}{3} }
.\]

Combining, we have
\[
	\frac{\Delta ^{2} \# P_{\Delta} ^{2} \# A ^{4} }{\# S} \lesssim \delta^{- O(\epsilon)}\mathrm{E} _{3,\delta} (A) \mathrm{E} _{3,\delta} (A,S_{\text{pop}} )^{\frac{1}{3} } \mathrm{E} _{\frac{3}{2}  ,\delta} (A,P_{\Delta} ')^{\frac{2}{3} }
.\]
Recall our energy bound
\[
	\mathrm{E} _{3,\delta} (A,- P_{\Delta} ')\lesssim \delta^{-O(\epsilon)} \frac{\delta^{-2s} N_{\delta} (AA)^{2} \# P_{\Delta} }{\Delta^{2}} 
.\]
Interpolating for the \(3/2\) energy using this, we see by H\"older that
\[
	\mathrm{E} _{\frac{3}{2} ,\delta} (A, - P_{\Delta} ') \leq \mathrm{E} _{1,\delta} (A , - P_{\Delta} ')^{\frac{3}{4} } \mathrm{E} _{3,\delta} (A , - P_{\Delta} ')^{\frac{1}{4} }
,\]
which upon substituting gives
\[
	\mathrm{E} _{\frac{3}{2} ,\delta} (A, -P_{\Delta} ')^{\frac{2}{3} } \lesssim \frac{ \delta^{- O(\epsilon)} \# A^{\frac{1}{2} } \# P_{\Delta} ^{\frac{2}{3} } \delta^{- s / 3} N_{\delta} (AA)^{\frac{1}{3} }}{\Delta^{\frac{1}{3} }} 
.\]
Using the bound on \(\mathrm{E} _{\frac{3}{2} , \delta} (A, - P_{\Delta} ')\) along with
\[
	\mathrm{E}_{3, \delta}\left(A, S_{\mathrm{pop}}\right) \lesssim \delta^{-O(\epsilon)} \frac{N_\delta(A A)^2 N_\delta\left(S_{\mathrm{pop}}\right)^3}{\# A \delta^{-\sigma}}
\]
and
\[
	\mathrm{E}_{3, \delta}(A) \lesssim \delta^{-O(\epsilon)} N_\delta(A A)^2 \# A,
\]
gives
\[
	\frac{\Delta ^{2} \# P_{\Delta} ^{2} \# A ^{4}}{\# S} \lesssim \delta^{- O(\epsilon)} \left( N_\delta(A A)^2 \# A \right) \left( \frac{N_\delta(A A)^2 N_\delta\left(S_{\mathrm{pop}}\right)^3}{\# A \delta^{-\sigma}} \right) ^{\frac{1}{3} }  \frac{ \# A^{\frac{1}{2} } \# P_{\Delta} ^{\frac{2}{3} } \delta^{- s / 3} N_{\delta} (AA)^{\frac{1}{3} }}{\Delta^{\frac{1}{3} }} 
\]
which upon simplifying is
\[
	\Delta ^{\frac{7}{3} } \# P_{\Delta} ^{\frac{4}{3} }  \lesssim \# S\# A ^{- \frac{17}{6} }  N_{\delta} (AA)^{3} N_{\delta} (S_{\text{pop}} )   \delta^{ - s / 3 + \sigma / 3   }
.\]

We have chosen our level \(\Delta\) so that
\[
	\Delta^{\frac{7}{3} } \# P_{\Delta} ^{\frac{4}{3} }= \left( \Delta^{\frac{7}{4} } \# P_{\Delta} \right) ^{\frac{4}{3} } \gtrsim _{\eta} \# A ^{- \eta} \mathrm{E} _{\frac{7}{4} ,\delta} (R_1,R_2)^{\frac{4}{3} }   
.\]
Interpoling for the second energy, we have by H\"older that
\[
	\mathrm{E} _{2,\delta}  (R_1,R_2) \leq \mathrm{E} _{\frac{7}{4} ,\delta} (R_1,R_2)^{\frac{4}{5} } \mathrm{E} _{3,\delta} (A)^{\frac{1}{5} }
,\]
so substituting gives
\[
	\mathrm{E} _{2,\delta}  (R_1,R_2) \lesssim_{\eta}  \# A ^{\eta} \left[  \# S\# A ^{- \frac{17}{6} }  N_{\delta} (AA)^{3} N_{\delta} (S_{\text{pop}} )   \delta^{ - s / 3 + \sigma / 3   } \right] ^{\frac{3}{5}  } \left( \delta^{-O(\epsilon)} N_\delta(A A)^2 \# A \right) ^{\frac{1}{5} }
\]
or by simplifying,
\[
	\mathrm{E} _{2,\delta}  (R_1,R_2) \lesssim_{\eta} \# A^{\eta} \delta^{- O(\epsilon)} \frac{\# S ^{\frac{3}{5} } N_{\delta} (AA) ^{\frac{11}{5} } N_{\delta} (S_{\text{pop}} )^{\frac{3}{5} } \delta ^{- s /5 + \sigma /5}}{\# A ^{\frac{3}{2} }} 
.\]
Using our lower bound on the second energy, we have
\[
	\frac{\# A ^{4 - \eta}}{\# S} \lesssim_{\eta}  \delta^{-O(\epsilon)} \frac{\# S ^{\frac{3}{5} } N_{\delta} (AA) ^{\frac{11}{5} } N_{\delta} (S_{\text{pop}} )^{\frac{3}{5} } \delta ^{- s /5 + \sigma /5}}{\# A ^{\frac{3}{2} }} 
,\]
which simplifies to 
\[
	\delta^{O(\epsilon)} \delta^{- \sigma / 5    } \# A ^{\frac{11}{2} - \eta }\lesssim_{\eta}  \delta^{- s /5} N_{\delta} (A+A) ^{\frac{11}{5} } N_{\delta} (AA)^{\frac{11}{5} }
,\]
and we are finished.
\end{proof}

We now complete the proof of Theorem \ref{thm.dissum}, which simply combines Lemma \ref{lem.sumlowerbounds} and Corollary \ref{cor:sum-energybounds}. For $\eta = \e/s$ we obtain
\[
    \delta^{O(\epsilon)} \delta^{- \sigma / 5    } \# A ^{\frac{11}{2}  }\lesssim \delta^{- s /5} N_{\delta} (A+A) ^{\frac{11}{5} } N_{\delta} (AA)^{\frac{11}{5} }
.\]
Let \(\upsilon>0\) be the largest so that

\[
N_\rho(A A) \leq \rho ^{- \upsilon} \text { for all } \delta<\rho<\delta^{O(\epsilon)} \text {. }
\]
Let \(\tau\) be such that \(\# S=\delta^{-\tau}\). By Corollary \ref{cor.fursten-sum}, we have \(\sigma \geq 5 s / 2-v\). Therefore,
\[
    \delta^{O(\epsilon)} \delta ^{- s / 2 + \upsilon / 5}\# A ^{\frac{11}{2} }\lesssim \delta^{- s / 5} N_{\delta} (A+A)^{\frac{11}{5} } N_{\delta} (AA)^{\frac{11}{5} }
,\]
and therefore
\[
    \frac{29}{5} s - O(\epsilon) \leq \frac{11}{5} \tau + \frac{12}{5} \upsilon
,\]
so one of \(\tau,\upsilon\) must be larger than \(29 s / 23 - O(\epsilon)\), completing the proof.

\bibliographystyle{alpha}
\bibliography{references}

\begin{thebibliography}{DW25}

\bibitem[Blo25]{bloom}
Thomas~F. Bloom.
\newblock Control and its applications in additive combinatorics.
\newblock {\em arXiv preprint arXiv:2501.09470}, 2025.

\bibitem[Bou03]{bou03}
J.~Bourgain.
\newblock On the {E}rd\"os-{V}olkmann and {K}atz-{T}ao ring conjectures.
\newblock {\em Geom. Funct. Anal.}, 13(2):334--365, 2003.

\bibitem[Cus25]{csumproduct}
Adam Cushman.
\newblock A note on the sum-product problem and the convex sumset problem.
\newblock {\em arXiv preprint arXiv:2512.13849}, 2025.

\bibitem[DO25]{doprod}
Ciprian Demeter and William O'Regan.
\newblock Incidence estimates for quasi-product sets and applications.
\newblock {\em https://arxiv.org/abs/2511.15899}, 2025.

\bibitem[DW25]{demwangszem}
Ciprian Demeter and Hong Wang.
\newblock Szemer\'edi-{T}rotter bounds for tubes and applications.
\newblock {\em Ars Inven. Anal.}, pages Paper No. 1, 46, 2025.

\bibitem[Ele97]{elekes1997number}
Gy{\"o}rgy Elekes.
\newblock On the number of sums and products.
\newblock {\em Acta Arithmetica}, 81(4):365--367, 1997.

\bibitem[EM03]{ed}
G.~A. Edgar and Chris Miller.
\newblock Borel subrings of the reals.
\newblock {\em Proc. Amer. Math. Soc.}, 131(4):1121--1129, 2003.

\bibitem[EV66]{erd}
Paul Erd\H{o}s and Bodo Volkmann.
\newblock Additive {G}ruppen mit vorgegebener {H}ausdorffscher {D}imension.
\newblock {\em J. Reine Angew. Math.}, 221:203--208, 1966.

\bibitem[Fal85]{falconerdistance}
K.~J. Falconer.
\newblock On the {H}ausdorff dimensions of distance sets.
\newblock {\em Mathematika}, 32(2):206--212, 1985.

\bibitem[KT01]{kat}
Nets~Hawk Katz and Terence Tao.
\newblock Some connections between {F}alconer's distance set conjecture and sets of {F}urstenburg type.
\newblock {\em New York J. Math.}, 7:149--187, 2001.

\bibitem[MO25]{mator}
Andr\'as M\'ath\'e and William O'Regan.
\newblock Discretised sum-product theorems by {S}hannon-type inequalities.
\newblock {\em J. Lond. Math. Soc. (2)}, 112(6):Paper No. e70389, 16, 2025.

\bibitem[OS23]{orpshabc}
Tuomas Orponen and Pablo Shmerkin.
\newblock Projections, {F}urstenberg sets, and the abc sum-product problem.
\newblock {\em arXiv preprint arXiv:2301.10199}, 2023.

\bibitem[RS22]{rud}
Misha Rudnev and Sophie Stevens.
\newblock An update on the sum-product problem.
\newblock {\em Math. Proc. Cambridge Philos. Soc.}, 173(2):411--430, 2022.

\bibitem[RW23]{renwang}
Kevin Ren and Hong Wang.
\newblock Furstenberg sets estimate in the plane.
\newblock {\em arXiv preprint arXiv:2308.08819}, 2023.

\bibitem[SS11]{schoenshkrsumsets}
Tomasz Schoen and Ilya~D. Shkredov.
\newblock On sumsets of convex sets.
\newblock {\em Combin. Probab. Comput.}, 20(5):793--798, 2011.

\end{thebibliography}

\end{document}